\documentclass[12pt]{article}

\usepackage{amsmath,amssymb,amsfonts}
\usepackage{amsthm}

\usepackage{enumerate}
\usepackage{picinpar}
\usepackage{layout}
\usepackage{verbatim}
\usepackage{epsfig}
\usepackage{graphicx}
\usepackage{pgf}
\newtheorem{theorem}{Theorem}[section]
\newtheorem{lemma}[theorem]{Lemma}
\newtheorem{corollary}[theorem]{Corollary}
\newtheorem{definition}[theorem]{Definition}

\newcounter{fig}

\def \eqdef {\overset{\text{\tiny def}}{=}}

\def\R{\mathbb R}
\def\Z{\mathbb Z}

\def\C{\mathcal C}
\def\S{\mathcal S}
\def\Re{\mathcal R}
\def\K{\mathcal K}

\title{Boundedness of trajectories for weakly reversible, single linkage class reaction systems}
\author{David F. Anderson$^{1}$}

\begin{document}

\maketitle

\footnotetext[1]{Department of Mathematics, University of Wisconsin, Madison,
  WI, 53706.  Grant support from NSF grant DMS-1009275}

\begin{abstract} 
  This paper is concerned with the dynamical properties of deterministically modeled chemical reaction systems with mass-action kinetics.  Such models are ubiquitously found in chemistry, population biology, and the burgeoning field of systems biology.  A basic question, whose answer remains largely unknown, is the following:  for which network structures do trajectories of mass-action systems remain bounded in time?    In this paper, we conjecture that the result holds when the reaction network is weakly reversible, and prove this conjecture in the case when the reaction network consists of a single linkage class, or connected component.   
  
\end{abstract}



\section{Introduction}

Building off the work of Fritz Horn, Roy Jackson, and Martin Feinberg
\cite{FeinbergLec79, Feinberg87, FeinHorn1972, Horn72, Horn74, HornJack72} the
mathematical theory termed ``Chemical Reaction Network Theory'' has, over the past 40 years, determined many of the basic qualitative properties of chemical reaction networks and, more generally, models of population processes.   As the exact values of key system parameters, termed \textit{rate constants} and which we will denote by $\kappa_k$, are usually difficult to find experimentally
and, hence, are oftentimes unknown, the results tend to be \textit{independent of the values of these parameters}.  In large part motivated by the Global Attractor Conjecture \cite{CraciunShiu09}, much of the recent attention in this field has focussed on which network structures guarantee that trajectories are \textit{persistent}, in that they can not approach the boundary of the positive orthant along a sequence of times \cite{Anderson2011, AndShiu, Sontag2007, Angeli2007, CraciunShiu09, CraciunPantea, JohnstonSiegel2, JohnstonSiegel2011, Pantea2011}.  In this paper we consider a related question:  for which network structures do trajectories of mass-action systems necessarily remain \textit{bounded} in time?  This question is similar to that of persistence in that both force us to consider extreme behaviors of the species, and, hence, the monomials of the dynamical system.  Similar to the well known Persistence Conjecture (see below), we will conjecture that all trajectories of weakly reversible systems with mass-action kinetics are bounded in time.  We will prove this conjecture in the case when the reaction network consists of a single linkage class, or connected component.  The methods used in this paper are similar to those introduced in \cite{Anderson2011}, where the Global Attracor Conjecture was shown to hold in the single linkage class case.

\subsection{Formal statement of the problem}
\label{sec:background}

Two of the most basic questions that can be asked about a mathematical model for a chemical, or more generally a population, process are $(i)$ must all trajectories be bounded in time and $(ii)$ are trajectories persistent in the sense of Definition \ref{def:persistence} below.

\begin{definition} 
  For $t\ge 0$ denoting time, let $\phi(t,x_0)$ be a trajectory to a dynamical system in $\R^N$ with initial condition $x_0$.  A trajectory $\phi(t,x_0)$ with state space $\R^N_{\ge 0}$  is said to be {\em 
    persistent} if 
    \begin{equation*}
	\liminf_{t\to \infty} \phi_i(t,x_0) > 0,
\end{equation*}
for all $i \in \{1,\dots,N\}$, where $\phi_i(t,x_0)$ denotes the $i$th component of $\phi(t,x_0)$. A dynamical system is  said to be
   {\em persistent} if each trajectory with non-negative initial 
  condition is persistent.  
  \label{def:persistence} 
\end{definition} 

 Thus, persistence corresponds to a non-extinction requirement.  Some authors refer to dynamical systems satisfying  the above condition  as \textit{strongly persistent} \cite{Takeuchi1996}.  In their work, persistence only requires the weaker condition that $\limsup_{t\to \infty} \phi_i(t,x_0) > 0$  for each $i\in\{1,\dots,N\}$. 

%

The following conjecture of Feinberg (see Remark 6.1.E in \cite{Feinberg87}) is one of the most well known in chemical reaction network theory.  It pertains to systems whose reaction networks are weakly reversible, or strongly connected (see Section \ref{sec:def_concepts}), and is intimately related to the Global Attractor Conjecture \cite{CraciunShiu09}.

\vspace{.1in}

\noindent \textbf{Persistence Conjecture.\ (Version 1)}  Any weakly reversible
reaction network with mass-action kinetics  is persistent.

\vspace{.1in}

In \cite{Anderson2011}, it was pointed out that there are really two natural conjectures pertaining to weakly reversible reaction networks with mass-action kinetics, and  that these should be separated.  

\begin{definition}
	For $t\ge 0$ denoting time, let $\phi(t,x_0)$ be a trajectory to a dynamical system in $\R^N$ with initial condition $x_0$.  A trajectory $\phi(t,x_0)$   is said to be {\em 
    bounded} if 
    \begin{equation*}
	\limsup_{t\to \infty} |\phi(t,x_0)| < \infty.
\end{equation*}
  A dynamical system is  said to have
   {\em bounded trajectories} if each trajectory  is bounded.  
\label{def:boundedTraj}
\end{definition}

\vspace{.1in}

\noindent \textbf{Persistence Conjecture.\ (Version 2)}  Any weakly reversible reaction network with mass-action kinetics and \textit{bounded trajectories} is persistent.

\vspace{.1in}

\noindent \textbf{Boundedness Conjecture.}  Any weakly reversible reaction network with mass-action kinetics has bounded trajectories.

\vspace{.1in}

Clearly, the latter two conjectures would imply the first, which would then imply the well known Global Attractor Conjecture (see \cite{CraciunShiu09, Feinberg87}).  Note that none of the conjectures make any assumptions on the choice
of rate constants, which are the natural parameters found in these systems (see Section \ref{sec:def_concepts}).  The Boundedness Conjecture stated above is quite similar to the Extended Permanence Conjecture found in \cite{CraciunPantea}, which conjectures that all ``endotactic'' systems (which include those that are weakly reversible) are \textit{permanent}.  Permanence is an even stronger condition than bounded trajectories in that all trajectories of a compatibility class (invariant manifold), regardless of initial condition, must enter a single compact subset of $\R^N_{>0}$.  The Extended Permanence Conjecture is proven in \cite{CraciunPantea} in the case when the system is two-dimensional.  In Section \ref{sec:permanence} we briefly discuss permanence and conclude that weakly reversible, single linkage class systems are permanent if there is a $\delta>0$ for which $\liminf_{t\to \infty}\phi_i(t,x_0) \ge \delta$ for all $x_0$ and all $i$.  That is, when the system is, in some sense, \textit{uniformly persistent}.

Each of the above mentioned conjectures remains open.  In recent years there has been a great amount of energy aimed at resolving the Persistence Conjecture, and typically that work has focused on Version 2.  This focus has been quite natural as much of the motivation for the work stemmed from consideration of  ``complex-balanced'' systems, see \cite{FeinbergLec79, Feinberg87}, which are known to have bounded trajectories.  Relatively little attention has been paid, therefore, to the related Boundedness Conjecture, as formally stated above.    We will refrain from giving an exhaustive background on the work aimed at resolving the Persistence Conjectures, and instead point the interested reader to \cite{Anderson2011}, where such an introduction, including most of the relevant references related to persistence and the Global Attractor Conjecture, can be found.

\subsection{Results in this paper}

  In this paper we will prove that the Boundedness Conjecture holds for all systems whose network consists of a single linkage class, or connected component (see Section \ref{sec:def_concepts}).   To prove our results, we will use a method, introduced in \cite{Anderson2011}, for partitioning the relevant monomials of the dynamical system along sequences of trajectory points into classes with comparable growths. This will allow us to prove that there is a Lyapunov function which decreases along all paths when $|x(t)|$ is sufficiently large. 
  
  We will prove all of our results in a slightly more general setting than mass-action kinetics in that we will allow our rate ``constants''  to actually be bounded functions of time.  Results pertaining to systems with such a generalized mass-action kinetics are useful as these systems arise naturally when a system with standard mass-action kinetics is {\em projected} onto a subset of the species (see Section 3 of \cite{Anderson2011}).
  
The outline of the paper is as follows.  In Section \ref{sec:def_concepts}, we provide a review of the requisite definitions and terminology from chemical reaction network theory.   In Section \ref{sec:results}, we give our main results together with their proofs.

\section{Preliminary concepts and definitions}
\label{sec:def_concepts}

Most of the following definitions are standard in chemical reaction network theory.  The interested reader should see  \cite{FeinbergLec79} or \cite{Gun2003} for a more detailed introduction.\vspace{.125in}

\noindent \textbf{Reaction networks.}  An example of a chemical reaction is $2S_{1}+S_{2} ~\rightarrow~ S_{3},$
where we interpret the above as saying two molecules of type $S_1$ combine with a molecule of type $S_2$ to produce a molecule of type $S_3$.  For now, assume that there are no other reactions under consideration.  The $S_{i}$ are called chemical {\em species} and the linear combinations of the species found at either end of the reaction arrow, namely $2S_{1}+S_{2}$ and
$S_{3}$, are called chemical {\em complexes.}  Assigning the {\em
  source} (or reactant) complex $2S_{1}+S_{2}$ to the vector $y =
(2,1,0)$ and the {\em product} complex $S_{3}$ to the vector
$y'=(0,0,1)$, we can formally write the reaction as $ y \rightarrow y' .$ 

 In the
general setting we denote the number of species by $N$, and for $i \in \{1,\dots, N\}$ we denote the $i$th species as $S_{i}$.  We then
consider a finite set of reactions with the $k$th denoted by $ y_{k} \rightarrow y_{k}', $
 where $y_k, y_k' \in \Z^N_{\ge 0}$ are (non-equal) vectors whose components give the coefficients of the source and product complexes, respectively.  Using a slight abuse of notation, we will also refer to the vectors $y_k$ and $y_k'$ as the complexes.  Note that if $y_k = \vec 0$ or $y_k' = \vec 0$ for some $k$, then the $k$th
reaction represents an input or output, respectively, to the system.  Note also that any
complex may appear as both a source complex and a product complex in
the system.  We will usually, though not always (for example, see condition 3 in Definition \ref{def:crn} below) use the prime $'$ to denote the product complex of a given reaction.

As an example, suppose that the entire network consists of the two species $S_1$ and $S_2$ and the two reactions
\begin{equation}
	S_1 \to S_2 \quad \text{and} \quad S_2 \to S_1,
	\label{eq:ex1}
\end{equation}
where $S_1 \to S_2$ is arbitrarily labeled as ``reaction 1.''  Then $N = 2$ and 
\begin{equation*}
	y_1 = (1,0), \quad y_1' = (0,1) \qquad \text{and} \qquad  y_2 = (0,1), \quad y_2' = (1,0).
\end{equation*}
Thus, the vector $(1,0)$, or equivalently the complex $S_1$, is both $y_1$, the source of the first reaction, and $y_2'$, the product of the second.

 For ease of notation, when there is no need for
enumeration we will typically drop the subscript $k$  from the notation
for the complexes and reactions.

\begin{definition}
  Let $\S = \{S_i\}_{i=1}^N$, $\C = \{y\}$ with $y \in \Z^N_{\ge 0}$, and $\Re = \{y \to y'\}$ denote
  finite sets of species, complexes, and reactions, respectively.  The triple
  $\{\S, \C, \Re\}$ is called a {\em chemical reaction network} so long as the following three natural requirements are met:
  \begin{enumerate}
     \item  For each $S_i\in \S$, there exists at least one complex $y\in \C$  for which $y_{i} \ge 1$.
     \item There is no trivial reaction $y \to y \in \Re$ for some complex $y \in \C$.
     \item For any $y\in \C$, there must exist a $y'\in \C$ for which $y \to y' \in \Re$ or $y' \to y \in \Re$.
  \end{enumerate}
  \label{def:crn}
\end{definition}

\textbf{Notation:}  We will use each of the following choices of notation to denote a complex from $\C$:  $y$, $y'$, $y_i$, $y_j$, $y_k$, $y_k'$, etc.  However, there will be other times in which we wish to denote the $i$th component of a complex.  If the complex in question has been denoted by $y_k$, then we will write $y_{k,i}$.  However, if the complex has been denoted by $y$, then we would write its $i$th component as $y_i$, which, through context, should not cause confusion with a choice of \textit{complex} $y_i$.  See, for example, condition 1 in Definition \ref{def:crn} above.

\begin{definition}
To each reaction network $\{\mathcal{S},\mathcal{C},\mathcal{R}\}$
we assign a unique directed graph called a {\em reaction diagram}
constructed in the following manner.  The nodes of the graph are the
complexes, $\mathcal{C}$.  A directed edge $(y,y')$ exists if and only
if $y \to y' \in \Re$.  Each connected
component of the resulting graph is termed a  {\em linkage class} of
the graph.  
\label{def:diagram}
\end{definition}

For example, the system described in and around \eqref{eq:ex1} has reaction diagram $S_1 \rightleftarrows S_2$,
which consists of a single linkage class.

\begin{definition}
  Let $\{\S,\C,\Re\}$ denote a chemical reaction network.  Denote the
  complexes of the $i$th linkage class by $L_i \subset \C$.  We say   $T \subset \C$
  consists of a  {\em union of linkage classes} if $T = \cup_{i \in I} L_i$
  for some nonempty index set $I$.
\end{definition}

\begin{definition}
  The chemical reaction network $\{\S,\C,\Re\}$ is said to be  {\em weakly
    reversible} if each linkage class of the corresponding reaction
  diagram is strongly connected.  A network is said to be
   {\em reversible} if $y' \to y \in \Re$ whenever $y \to y' \in
  \Re.$ 
  \label{def:WR}
\end{definition}

 It is easy to see that a chemical reaction network is weakly reversible if and only if for each reaction $y \to y'\in \Re$, there exists a sequence of complexes, $y_1,\dots, y_r\in \C$, such that $y' \to y_1 \in \Re, y_1 \to y_2 \in \Re, \cdots, y_{r-1}\to y_r\in \Re,$ and $y_r \to y\in \Re$.

\vspace{.225in}

\noindent \textbf{Dynamics.}  A chemical reaction network gives rise to a dynamical system by way of
a \textit{rate function} for each reaction.
That is, for each $y_k \to y_k'\in \Re$, or simply $k\in\{1,\dots,|\Re|\},$ we suppose the existence
of a function $\displaystyle R_k =
R_{y_k \to y_k'}$ that determines the rate of that reaction.
 The functions $R_{k}$ are
typically referred to as the \textit{kinetics} of the system and will be denoted by $\K$, or $\K(t)$ in the non-autonomous case.  The
dynamics of the system is then given by the following coupled set of
(typically nonlinear) ordinary differential equations
\begin{equation}
  \dot x(t) = \sum_{k} R_{k}(x(t),t)(y_k' - y_k),
  \label{eq:main_general}
\end{equation}
where $k$ enumerates over the reactions and $x(t) \in \R^N_{\ge 0}$ is a vector whose $i$th component represents the concentration of species $S_i$ at time $t\ge 0$.

\begin{definition}
	A chemical reaction network $\{\S,\C,\Re\}$ together with a choice of kinetics $\K$ is called a {\em chemical reaction system} and is denoted via the quadruple $\{\S,\C,\Re,\K\}$.  In the non-autonomous case where the $R_k$ can depend explicitly on $t$, we will write $\{\S,\C,\Re,\K(t)\}$.   We say that a chemical reaction system is
   {\em weakly reversible} if its underlying network is.
\end{definition}

Integrating \eqref{eq:main_general} with respect to time yields
\begin{equation*}
  x(t) = x(0) + \sum_{k} \left(\int_0^t R_k(x(s),s) ds \right)
  (y_k' - y_k). 
\end{equation*}
Therefore, $x(t) - x(0)$ remains within $S =
\text{span}\{y_k' - y_k\}_{k \in \{1,\dots,R\}}$ for all time.

\begin{definition}
  The  {\em stoichiometric subspace} of a network is the linear
  space $S = \text{\em span}\{y_k' - y_k\}_{k \in \{1,\dots,|\Re|\}}$.  The vectors $y_k' - y_k$ are called the  {\em reaction vectors}.
  \label{def:stoich_sub}
\end{definition}

Under mild conditions on the rate functions of a
system, a trajectory $x(t)$ with
strictly positive initial condition $x(0) \in \R^N_{>0}$ remains in the
strictly positive orthant $\R^N_{>0}$ for all time (see, for example, Lemma~2.1 of
\cite{Sontag2001}).  Thus, the trajectory remains in the relatively open set $(x(0)
+ S) \cap \mathbb{R}^N_{> 0}$, where $x(0) + S := \{z \in \R^N \ | \ z
= x(0) + v, \text{ for some } v \in S\}$, for all time.  In other
words, this set is \textit{forward-invariant} with respect to the
dynamics.  It is also easy to show that under the same mild conditions on $R_k$, $(x(0) + S) \cap
\mathbb{R}^N_{\ge  0}$ is forward invariant with respect to the dynamics.  The sets $(x(0) + S) \cap
\mathbb{R}^N_{> 0}$ 
will be referred to as the \textit{positive stoichiometric compatibility classes}, or simply as the \textit{positive classes}.  Note that if each of the sets $(x(0) + S) \cap \mathbb{R}^N_{> 0}$ is bounded, 
then all trajectories of the dynamical system are necessarily bounded also.  Therefore, the results of this paper are of interest when each positive class is an unbounded set.

The most common kinetics is that of \textit{mass-action kinetics}. A
chemical reaction system is said to have mass-action kinetics if all rate functions $R_{k} = R_{y_k \to y_k'}$ 
take the multiplicative form
\begin{equation}
  R_{k}(x) =  \kappa_k x_1^{y_{k,1}} x_2^{y_{k,2}} \cdots x_N^{y_{k,N}},
  \label{eq:massaction}
\end{equation}
where $\kappa_k$ is a positive reaction rate constant and $y_k$ is the source complex for the reaction.  For $u\in \R^N_{\ge 0}$ and $v \in \R^N$, we define 
\begin{equation*}
 u^v \eqdef u_1^{v_1} \cdots u_N^{v_N},
\end{equation*}
 where we have
adopted the convention that $0^0 = 1$, and the above is undefined if $u_i = 0$ when $v_i < 0$.  Mass action kinetics can then be written succinctly as $R_k(x) = \kappa_k x^{y_k}.$
Combining \eqref{eq:main_general} and \eqref{eq:massaction} gives the
following system of differential equations  
\begin{equation}
  \dot x(t) = \sum_{k} \kappa_k x(t)^{y_k}(y_k' - y_k).
  \label{eq:main}
\end{equation}
  
We will generalize the equation \eqref{eq:main} slightly by allowing each $\kappa_k$ to be a bounded function of time. See Definition 2.6 in \cite{CraciunPantea} for a similar definition, and \cite{Angeli2011} for another recent treatment of chemical reaction systems with non-autonomous dynamics.

\begin{definition}
	We say that the non-autonomous system $\{\S,\C,\Re,\K(t)\}$ has  {\em bounded mass-action kinetics} if there exists an $\eta > 0$ such that for each $k\in \{1,\dots,|\Re|\}$
	\begin{equation*}
		R_k(x,t) = \kappa_k(t) x^{y_k},
	\end{equation*}
	where $\eta < \kappa_k(t) < 1/\eta$ for all  $t \ge 0$.  Hence, the vector of concentrations satisfies
	\begin{equation*}
  \dot x(t) = \sum_{k} \kappa_k(t) x(t)^{y_k}(y_k' - y_k).
  \label{eq:main2}
	\end{equation*}
\end{definition}

 We require some final notation.  Let  $v \in \R^N$ for some $N \ge 1$, and let $U \subset \{1,\dots,N\}$.  We write $U[j]$ for the $j$th component of $U$. We then write $v|_U$ to denote the vector of size $|U|$ with 
\begin{equation*}
 v|_{U,j} = (v|_U)_j \eqdef v_{U[j]}
\end{equation*}
 for $j \in \{1,\dots, |U|\}.$
  Thus, $v|_U$ simply denotes the projection of $v$ onto the components enumerated by $U$.
For example, if $N = 8$ and $U = \{2,4,7\}$, then for any $v \in \R^8$, $v|_{U} = (v_2,v_4,v_7)\in \R^3$.

\section{Main results}
\label{sec:results}

 Recall that for any vectors $u,v$ such that $u \in \R^N_{\ge 0}$ and $v\in \R^N$ we define $u^v \eqdef u_1^{v_1}\cdots u_N^{v_N}$, where we use the convention $0^0 = 1$.   For completeness, we recall the following standard definition.
 
 \begin{definition}
 	For any set $\C$, we say $\{T_i\}$ is a {\em partition} of $\C$ if each $T_i$ is non-empty, $\bigcup_i T_i = \C$, and for all $i\ne j$,  $T_i \cap T_j = \emptyset$.
 \end{definition}
 

 The following combination of Definition \ref{def:partition} and Lemma \ref{lem:partition} is a generalization of Definition 4.1 and Lemma 4.2 found in  \cite{Anderson2011}.  While the generalization is not made use of in the current paper, we hope that pointing out that the function $f$ below can be nearly arbitrary will be beneficial in future work.

\begin{definition}
  Let $\C$ denote a finite set of vectors in $\R^N$.  Let $x_n \in
  \R^N$ denote a sequence of points.  For $D \subset \R^N$ with $\{x_n\} \subset D$, let $f:D\times \C \to \R_{>0}$.  We say that $\C$ is
  {\em partitioned along the sequence $\{x_n\}$ with respect to $f$} if there exists a partition, $\{T_i\}_{i = 1}^P$, of $\C$, where the $T_i$ are termed tiers, and a constant $C > 1$,  such that
  \begin{enumerate}[$(i)$]
  \item if $y_j,y_k \in T_i$ for some $i \in \{1,\dots,P\}$, then for all $n$
    \begin{equation*}
      \frac{1}{C} f(x_n,y_j) \le f(x_n,y_k) \le C f(x_n,y_j).
    \end{equation*}

  \item if $y_j \in T_i$ and $y_k \in T_{i+m}$ for some $m \ge 1$,
    then
    \begin{equation*}
      \frac{f(x_n,y_j)}{f(x_n,y_k)}  \to \infty, \quad
      \text{as } n \to \infty.
    \end{equation*}
  \end{enumerate}
 When $f(x,y) = x^y$, the case considered in both this paper and \cite{Anderson2011}, we will simply say that {\em $\C$ is partitioned along $\{x_n\}$}.
 \label{def:partition}
\end{definition}
Note that we have a natural ordering of the tiers: $T_1 \succ T_2 \succ T_3 \succ \cdots \succ T_P$, and we say $T_1$ is the ``highest'' tier, whereas $T_P$ is the ``lowest'' tier.

The proof of the following lemma is a slight modification of the proof of Lemma 4.3 in \cite{Anderson2011} and is omitted. 

\begin{lemma}
  Let $\C$ denote a finite set of vectors in $\R^N$.  Let $x_n$ be a
  sequence of points in $\R^N_{>0}$.  For $D \subset \R^N$ with $\{x_n\} \subset D$, let $f:D\times \C \to \R_{>0}$.  Then, there exists a subsequence of
  $\{x_n\}$ along which $\C$ is partitioned with respect to $f$.
  \label{lem:partition}
\end{lemma}

The following lemma, which is similar in spirit to Farkas' Lemma, states that for any set of vectors in $\R^N$, either their span includes a non-zero vector in the non-positive orthant $\R^N_{\le 0}$, or there is vector normal to their span that intersects the strictly positive orthant.

\begin{lemma}[Stiemke's Theorem, \cite{Stiemke1915}]
  For $i = 1,\dots, n$, let $u_i \in \R^m$. Either there exists a $c \in \R^n$ such that 
  \begin{equation*}
    \left(\sum_{i=1}^n c_i u_i\right)_j \le  0, \quad j = 1,\dots,m
  \end{equation*}
  and such that at least one of the inequalities is strict,
  or there is a $w \in \R^m_{> 0}$ such that $w\cdot u_i=0$ for each
  $i\in \{1,\dots, n\}$.
  \label{lem:Stiemke}
\end{lemma}

\begin{corollary}
    For $i = 1,\dots, n$, let $u_i \in \R^m$. Let $U \subset \{1,\dots, m\}$ and $V = U^c$.  Either there exists a $c \in \R^n$ such that
  \begin{align*}
    \left(\sum_{i=1}^n c_i u_i\right)_j &\le  0, \quad j \in U\\
        \left(\sum_{i=1}^n c_i u_i\right)_j &\ge  0, \quad j \in V
  \end{align*}
  and such that at least one of the inequalities is strict, or there is a $w \in \R^m$ with
  \begin{align*}
  	w_j &> 0 \text{ for } j \in U\\
	w_j & < 0 \text{ for } j \in V
  \end{align*}
   such that $w\cdot u_i=0$ for each
  $i\in \{1,\dots,n\}$.
  \label{lem:Stiemke2}
\end{corollary}
\begin{proof}
  Define the vector valued function  $\theta:\R^m \to \R^m$ via
\begin{equation*}
 \theta(u)_j \eqdef \left\{ \begin{array}{cl}
    u_j & \text{ if } j \in U\\
    -u_j & \text{ if } j \in V
 \end{array}\right. .
\end{equation*}
Applying Lemma \ref{lem:Stiemke} to the set of vectors $\theta(u_i)$ proves the result.
\end{proof}
%


\begin{definition}
	Let $w \in \R^N$.  The set $\{i \in \{1,\dots,N\} \ : \ w_i>0\}$ is called the {\em positive support} of $w$ and the set $\{i \in \{1,\dots,N\} \ : \ w_i<0\}$ is called the {\em negative support} of $w$.  The union of the positive and negative support of $w$, i.e. the set $\{i \in \{1,\dots,N\} \ : \ w_i \ne 0\}$, is called the {\em support} of $w$. 
\label{def:support}
\end{definition}

\begin{definition}
  Let $\C$ denote a finite set of vectors in $\R^N$. Let $\{T_i\}$
  denote a partition of $\C$.  Let $U,V \subset \{1,\dots,
  N\}$ with $U\cup V$ nonempty.  We will say that the vector $w \in \R^N$ is a {\em conservation relation that respects the triple} $(U,V,\{T_i\})$ if the
  following two conditions hold:
  \begin{enumerate}
  \item $U$ is the positive support of $w$ and $V$ is the negative support of $w$. 
  \item Whenever $y_j,y_{\ell} \in T_i$
    for some $i$, we have that $w \cdot (y_j - y_{\ell}) = 0$.
  \end{enumerate}
\end{definition}

%

\begin{definition} 
	Let $x_n \in \R^N_{>0}$ denote a sequence of points.  We say that $x_n$ is {\em partially monotonic} if $x_{n,i} \ge x_{n+1,i}$ for each $i$ for which $\liminf_{n\to \infty} x_{n,i}=0$ and if $x_{n,j} \le x_{n+1,j}$ for each $j$ for which $\limsup_{n\to \infty} x_{n,j}=\infty$.
	\label{def:partial_mon}
\end{definition}

Note that Definition \ref{def:partial_mon} stands silent on the behavior of those $j$ for which $0< \liminf_{n\to \infty} x_{n,j} \le \limsup_{n\to\infty} x_{n,j} < \infty$.

\begin{theorem}

  Let $\C$ denote a finite set of vectors in $\R^N$.  Let $x_n \in  \R^N_{>0}$ denote a partially monotonic sequence of points for which $\lim_{n \to \infty}x_{n,i}  \in \{0,\infty\}$ for at least one $i \in \{1,\dots,N\}$.
 Let 
  \begin{align*}
   U &=  \{i \in \{1,\dots,N\} \, : \, \lim_{n\to \infty} x_{n,i} = 0\}\\
   V &=  \{j \in \{1,\dots,N\} \, : \, \lim_{n\to \infty} x_{n,j} = \infty\}.
  \end{align*}
   Finally, suppose that $\C$ is partitioned along
  $\{x_{n}\}$ with tiers $T_i$, for $i=1,\dots,
  P$, and constant $C>0$.  Then,  there is a
  conservation relation $w \in \R^N$ that
  respects the triple $(U, V,\{T_i\})$. 
  \label{thm:conservation}
\end{theorem}

\begin{proof}
 We suppose, in order to find a contradiction, that there is no
  conservation relation that respects the triple $(U, V,\{T_i\})$.  
    Define the sets $W_i \subset \R^N$, for $i = 1,\dots,P$, and $W \subset \R^N$ via
  \begin{align*}
    W_i &\eqdef \{y_{j} - y_{k}\in \R^N\ | \  y_{j},  y_{k} \in T_i \},\quad 
    W \eqdef  \bigcup_{i = 1}^{P} W_i,
  \end{align*}
  and denote the elements of $W$ by $\{u_k\}$.  Note that if $T_i$ consists of a single element, then $W_i$ consists solely of the zero vector.  Let $m = |U\cup V| >0$ be the number of elements
  in $U\cup V$ and let 
  $W|_{U\cup V} \subset \R^m$
   be the restriction 
   of $W$ to the
  components associated with the index set $U\cup V$, as discussed at the end of Section \ref{sec:def_concepts}.  Denote the elements of $W|_{U\cup V}$ by
  $\{v_k\}$.  Thus, collecting terminology, $u_k \in \R^N$, whereas $v_k\in \R^m$, and for each $u_k\in W$, there is a
  corresponding $v_k \in W|_{U\cup V}$ for which $u_k|_{U\cup V} = v_k$, however the  mapping $\cdot |_{U\cup V}$ need not be injective.

  The set $W|_{U\cup V}$ must contain at least one nonzero vector because otherwise any non-negative vector with support $U\cup V$ would be a non-negative conservation relation that respects the triple $(U,V,\{T_i\})$, but we have assumed that no such relation exists.

    Because we have assumed
  there is no conservation relation that respects the triple
  $(U, V,\{T_i\})$, we may conclude by Corollary \ref{lem:Stiemke2} that there exist $c_{k} \in \R$ such that
  \begin{align}
  \begin{split}
    \left(\sum_{v_k \in W|_{U\cup V}}c_{k}
      v_{k}\right)_j &\le 0, \quad \text{ if } j \in U\\
      \left(\sum_{v_k \in W|_{U\cup V}}c_{k}
      v_{k}\right)_j &\ge 0, \quad \text{ if } j \in V,
      \end{split}
    \label{app:stiemke1}
  \end{align}
 and such that the inequality is strict for
  at least one $j\in U\cup V$.

  For $v_k \in W|_{U\cup V}$, let $m_{k}$ denote the number of vectors of $W$ that reduced to it.  Define the function $M:\R^N \to \R$ by
  \begin{equation*}
    M(x) \eqdef   \left[ \prod_{ u_k \in W }
      \left(x^{u_k} \right)^{c_{k}/m_{k}} \right],
  \end{equation*}
  where $c_{k}$ and $m_{k}$ are chosen for $u_k \in W$ if $u_k|_{U\cup V} = v_k\in W|_{U\cup V}$.  
  Note that, by construction and by the definition of partitioning along a sequence, if $u_k \in W$, then there are $y_j,y_{\ell} \in T_i$ for some $i$, such that $u_k = y_{\ell} - y_{j}$ and 
  \begin{equation*}
    \frac{1}{C} \le x_{n}^{u_k} =  \frac{x_{n}^{y_{\ell}}}{x_{n}^{y_j}} \le C,
  \end{equation*}
  for all $n\ge 1$.  Therefore, $M(x_{n})$ is uniformly, in $n$, bounded both from 
  above and below.    Noting that each $x_n$ has strictly positive components, we may take logarithms and find
  \begin{align}
    \ln(M(x_{n})) = \bigg(  \sum_{u_k\in W}
    \frac{c_{k}}{m_{k}} u_k \bigg) \cdot \ln x_{n},
    \label{eq:logM}
  \end{align}
  where for a vector $u \in \R^N_{>0}$ we define
  \begin{align*}
  	\ln(u) \eqdef (\ln(u_1),\cdots, \ln(u_N)).
  \end{align*}
  Expanding equation \eqref{eq:logM} along elements of $U\cup V$ and $(U\cup V)^c$ yields,
  \begin{align}
  \begin{split}
     \ln(M(x_{n})) &= \bigg(  \sum_{v_k\in W|_{U\cup V}} c_{k} v_k|_U \bigg) \cdot \ln (x_{n}|_U) + \bigg( \sum_{v_k\in W|_{U\cup V}} c_{k} v_k|_V \bigg) \cdot \ln (x_{n}|_V)\\
     &\hspace{.2in} +    \bigg(  \sum_{u_k\in W} \frac{c_{k}}{m_{k}} u_k|_{(U\cup V)^c} \bigg) \cdot \ln (x_{n}|_{(U\cup V)^c}).
     \end{split}
     \label{eq:bound}
  \end{align}
  By construction, $x_{n,\ell}$ is bounded from both above and below
  for $\ell \in (U\cup V)^c$. Thus, the final term in \eqref{eq:bound} is bounded from above and below. By the inequalities in \eqref{app:stiemke1}, where at least one term is strict, and the facts that $x_{n,i} \to 0$ for each $i \in U$ and  $x_{n,i} \to \infty$ for each $j \in V$ along this subsequence, we may conclude that the sum of the first and second term, and hence $\ln(M(x_{n}))$ itself, is unbounded towards positive infinity as $n \to \infty$.  This is a
  contradiction with the previously found fact that $M(x_{n})$ is
  uniformly bounded above and below, and the result is shown.
\end{proof}

\subsection{Bounded trajectories in the single linkage class case}
\label{sec:results2}

Define $V_1:\R^N_{>0} \to \R_{\ge 0}$  by
\begin{equation}
  V_1(z) \eqdef  \sum_{i = 1}^N \left[ z_i (\ln(z_i) - 1) + 1\right].
  \label{eq:Lyapunov} 
\end{equation}
This is the standard Lyapunov function of chemical reaction network theory where we have chosen $\overline x = (1,\dots,1)$ \cite{FeinbergLec79, Gun2003}.   Note that $\nabla V_1(x) = \ln x$. 
It is straightforward to show that  $V_1$ is convex with a global minimum of zero at $(1,\dots,1)$ \cite{FeinbergLec79}.    The following is a generalization of Lemma 4.7 in \cite{Anderson2011}.

\vspace{.125in}

\begin{lemma}
  Let $\{\S,\C,\Re,\K(t)\}$, with $\S = \{S_1,\dots, S_N\}$, be a weakly reversible,
  non-autonomous mass-action system with bounded kinetics.   Let $D \subset \R^N_{>0}$.  
  One of the following two conditions holds:
  \begin{enumerate}
  \item[C1:] There exists an $M>0$, such that for any $x \in D$ for which $x_{i} > M$ or $x_i < 1/M$ for at least one $i\in \{1,\dots,N\}$, we have
  	\begin{equation*}
      \sum_k \kappa_k(t) x^{y_k}(y_k' - y_k)\cdot \ln(x) < 0, \quad \text{for all } t \ge 0.
  \end{equation*}
      
  \item[C2:] There exists a sequence of points $x_n \in D$ for which  $\lim_{n \to \infty} x_{n,i} \in\{0, \infty\}$ for at least one $i$ and
    \begin{enumerate}[$(i)$]
    \item $\C$ is partitioned along $x_n$ with tiers
      $\{T_i\}_{i=1}^P$, and constant $C$, and
    \item $T_1$ consists of a union of linkage classes.
    \end{enumerate}
  \end{enumerate}
  \label{lem:main}
\end{lemma}

\begin{proof}
  We suppose condition $C1$ does not hold, and will conclude that condition $C2$ must then hold.  Because condition $C1$ does not hold, there is a sequence of points $x_n \in D$ and times $t_n\ge 0$ for which $\lim_{n\to \infty} x_{n,i} \in\{0,\infty\}$ for at least one $i$ and 
  \begin{equation}
      \sum_k \kappa_k(t_n) x_n^{y_k}(y_k' - y_k)\cdot \ln(x_n)\ge 0.
      \label{eq:bound1}
  \end{equation}
 Applying Lemma \ref{lem:partition}, we partition the complexes along an
  appropriate subsequence of $\{x_n\}$ with tiers $T_i$, $i = 1,\dots,P$, and constant $C>1$.    Note that this also had the effect of only considering the analogous subsequence of $\{t_n\}$.
  
  In the following, for tier $i\in \{1,\dots,P\}$, we denote by 
  \begin{itemize}
  \item $\{i \to i\}$ all reactions with both source and product complex in $T_i$,
  \item $\{i \to i + m\}$ all reactions with source complex in $T_i$ and product complex in $T_{i+m}$ for $m \ge 1$,
  \item $\{i \to i - m\}$ all reactions with source complex in $T_i$ and product complex in $T_{i-m}$ for $m \ge 1$.
  \end{itemize}
  Defining $u/v \eqdef (u_1/v_1,\dots, u_N/v_N)$ for $u,v \in \R^N_{>0}$, we re-write the left hand side of the inequality \eqref{eq:bound1}
  \begin{align}
    \sum_k \kappa_k(t_n) x_{n}^{y_k}(y_k' - y_k)\cdot \ln (x_n) 
    &= \sum_{i = 1}^P\bigg[ \sum_{\{ i \to i \}} \kappa_k(t_n)
    x_{n}^{y_k} \ln \left( \frac{x_n^{y_k'}} {x_n^{y_k}}\right) \label{eq:1} \\
    &\hspace{.2in}+ \sum_{m=1}^{P-i} \sum_{\{ i \to i + m\}} \kappa_k(t_n)
    x_{n}^{y_k} \ln \left( \frac{x_n^{y_k'}} {x_n^{y_k}}\right)\label{eq:2} \\
    &\hspace{.2in} + \sum_{m=1}^{i-1} \sum_{\{i \to i - m\}} \kappa_k(t_n)
    x_{n}^{y_k}\ln \left( \frac{x_n^{y_k'}} {x_n^{y_k}}\right)\bigg].\notag
  \end{align}
   Note that, by construction, for large enough $n$ any component in the enumeration \eqref{eq:2} is negative, and, in fact, $\ln(x_n^{y_k'}/x_n^{y_k}) \to -\infty$ as $n \to \infty$, for these terms.   The proof that the total summation above (that is, the left hand side of \eqref{eq:1}) must also, for large enough $n$, be strictly negative unless condition $C2$ holds is now identical to the analogous portion of the proof of Lemma 4.7 in \cite{Anderson2011}, and is omitted here.
\end{proof}

  \begin{lemma}
  	 Let $\{\S,\C,\Re\}$, with $\S = \{S_1,\dots, S_N\}$, be a single linkage class chemical reaction network.  
  Then, there does not exist a sequence of points $x_n\in \R^N_{>0}$, all in the same stoichiometric compatibility class,  for which  $\lim_{n \to \infty} x_{n,i} \in\{0, \infty\}$ for at least one $i$ and
    \begin{enumerate}[$(i)$]
    \item $\C$ is partitioned along $x_n$ with tiers
      $\{T_i\}_{i=1}^P$, and constant $C$, and
    \item $T_1$ consists of a union of linkage classes.
    \end{enumerate}
    \label{lem:no_union}
  \end{lemma}
  
  \begin{proof}
  	Note that in the one linkage class case $T_1$ can only consist of a union of linkage classes if $T_1 \equiv \C$. We suppose, in order to find a contradiction, that there is a sequence, $\{x_n\}$, all in the same stoichiometric compatibility class,  for which  $\lim_{n \to \infty} x_{n,i} \in\{0, \infty\}$ for at least one $i$ and
    \begin{enumerate}[$(i)$]
    \item $\C$ is partitioned along $x_n$ with tiers
      $\{T_i\}_{i=1}^P$, and constant $C$, and
    \item $T_1$ consists of a union of linkage classes.
    \end{enumerate}
    Perhaps after restricting ourselves to a sub-sequence, we may choose $x_n$ to be partially monotonic (recall Definition \ref{def:partial_mon}). Let 
  \begin{align*}
   U &=  \{i \in \{1,\dots,N\} \, : \, \lim_{n\to \infty} x_{n,i} = 0\}\\
   V &=  \{j \in \{1,\dots,N\} \, : \, \lim_{n\to \infty} x_{n,j} = \infty\}.
  \end{align*}  
  Note that $U\cup V$ is nonempty by construction.  By Theorem \ref{thm:conservation} there is a conservation relation $w \in \R^N_{\ge 0}$ that
  respects the triple $(U,V,\{T_i\})$.  
  
 For each $j \in V$,    $w_j < 0$ and $x_{n,j}\to \infty$.  Thus, if $V$ is nonempty, $w\cdot x_n \to -\infty$, as $n\to \infty$.  If $V$ is empty, then $U$ is necessarily nonempty and, by construction, $w\cdot x_n \to 0$, as $n \to \infty$.  However, because $T_1 \equiv \C$, we have that $w \cdot (y_k' - y_k) = 0$ for all $y_k \to y_k' \in \Re$.  Thus, as the $x_n$ are all in the same stoichiometric compatibility class, we have that $w\cdot x_n$ is a constant.  This shows that we can not have $w\cdot x_n \to -\infty$, as $n \to \infty$, and so $V$ must be empty.  However, by our construction we may then conclude that $w_i \ge 0$ for all $i$, and $w_i>0$ for at least one $i$.  Hence, $w\cdot x_n > 0$, and not zero.
  \end{proof}

  We now have our main result.
  
  \begin{theorem}
       Let $\{\S,\C,\Re,\K(t)\}$, with $\S = \{S_1,\dots, S_N\}$, be a single linkage class, weakly reversible,
  non-autonomous mass-action system with bounded kinetics.  Then, $\limsup_{t\to \infty} |\phi(t,x_0)| < \infty$ for  each $x_0 \in \R^N_{>0}$.   That is, the system has {\em bounded trajectories.}
  \label{thm:main1}
  \end{theorem}
  
\begin{proof}  
  Letting $D$ be a non-empty positive stoichiometric compatibility class, in the statement of Lemma \ref{lem:main}, we conclude by combining Lemmas \ref{lem:main} and \ref{lem:no_union}  that there is an $M>0$ so that for any $x\in D$ with $|x|>M$, we have 
  	\begin{equation*}
     \sum_k \kappa_k(t) x^{y_k}(y_k' - y_k)\cdot \ln(x) < 0, \quad \text{for all } t\ge 0.
  \end{equation*}
  Therefore, 
  \begin{equation}
  	\frac{\partial}{\partial t} V_1(\phi(t,x_0)) < 0, 
	\label{eq:good_bound}
  \end{equation}
  whenever $|\phi(t,x_0)|>M$.  Let $B_{x_0} = \sup\{V_1(x)\ : \ |x| = M\text{ or } x = x_0\}$.
    Inequality \eqref{eq:good_bound} shows that $V_1(\phi(t,x_0)) \le B_{x_0}$ for all $t\ge 0$, which when combined with the fact that $V_1(x) \to \infty$ as $|x| \to \infty$, 
    proves the result.
\end{proof}

%


%

\subsection{Permanence}
\label{sec:permanence}
Note that Lemma \ref{lem:main} and, in particular, equation \eqref{eq:good_bound} do not give a rate at which $V_1$ is decreasing.  Hence, we can not conclude in general that all trajectories contained within a given stoichiometric compatibility class enter a single compact subset of $\R^N_{\ge 0}$. That is, we can not conclude that trajectories are permanent  in the sense of Definition \ref{def:permanent} below.  This is quantified above by the explicit dependence of $B_{x_0}$ upon $x_0$. However, we may  strengthen our results slightly.

\begin{definition}
   For $t\ge 0$ denoting time, let $\phi(t,x_0)$ be a trajectory to a dynamical system in $\R^N$ with initial condition $x_0$.  The system is said to be {\em permanent} if there is a $\rho > 0$ such that for every $x_0\in \R^N_{\ge 0}$,  
   \begin{equation*}
 \rho < \liminf_{t\to \infty} \phi_i(t,x_0) \le \limsup_{t\to \infty} \phi_i(t,x_0) < 1/\rho
    \end{equation*}
     for all  $i \in \{1,\dots,N\}$. 
\label{def:permanent}
\end{definition}

\begin{lemma}
  Let $\{\S,\C,\Re,\K(t)\}$, with $\S = \{S_1,\dots, S_N\}$, be a weakly reversible,
  non-autonomous mass-action system with bounded kinetics.   Let $D \subset \R^N_{>0}$ be such that {\em dist}$(D,\partial \R^N_{>0}) >\delta$, for some $\delta>0$.  
  One of the following two conditions holds:
  \begin{enumerate}
  \item[C1:] For any $\epsilon > 0$, there exists an $M=M_{\epsilon,\delta}>0$ such that for any $x \in D$ with $|x|>M$, we have
  	\begin{equation*}
      \sum_k \kappa_k(t) x^{y_k}(y_k' - y_k)\cdot \ln(x) < - \epsilon, \quad \text{for all } t \ge 0.
  \end{equation*}
      
  \item[C2:] There exists a sequence of points $x_n \in D$ that satisfies $\lim_{n \to \infty} |x_n| = \infty$ and
    \begin{enumerate}[$(i)$]
    \item $\C$ is partitioned along $x_n$ with tiers
      $\{T_i\}_{i=1}^P$, and constant $C$, and
    \item $T_1$ consists of a union of linkage classes.
    \end{enumerate}
  \end{enumerate}
  \label{lem:main2}
\end{lemma}

\begin{proof}
	The proof is essentially the same as for Lemma \ref{lem:main}.  We first suppose condition $C1$ does not hold.  Let $\epsilon > 0$.  By our assumption, there must be a sequence of points $x_n \in D$ and times $t_n\ge 0$ such that $\lim_{n\to \infty}|x_n| = \infty$ and 
\begin{equation*}
	\sum_k \kappa_k(t_n) x_n^{y_k}(y_k' - y_k)\cdot \ln(x_n) \ge -\epsilon.
\end{equation*}
The proof is now exactly the same as for Lemma \ref{lem:main}, except that you recognize that for any $y\in T_1$, we necessarily have that $x_n^y \to \infty$, as $n \to \infty$.  Therefore, the terms in the summation in \eqref{eq:2} not only  dominate the others, but force the expression to $-\infty$ as $n\to \infty$, thereby concluding the proof.
\end{proof}

\begin{corollary}
    Let $\{\S,\C,\Re,\K(t)\}$, with $\S = \{S_1,\dots, S_N\}$, be a single linkage class, weakly reversible,
  non-autonomous mass-action system with bounded kinetics.  Let $P$ be a positive stoichiometric compatibility class and suppose there is a $\delta>0$ so that 
  \begin{equation*}
     \liminf_{t \to \infty}\phi_i(t,x_0) > \delta, \quad  \text{for all } i \in \{1,\dots,N\} \text{ and all } x_0 \in P.
     \end{equation*}     
  Then,  there is a $\rho > 0$ such that for any $x_0 \in P$ 
  \begin{equation*}
    \rho < \liminf_{t\to \infty} \phi_i(t,x_0) \le \limsup_{t\to \infty} \phi_i(t,x_0) < 1/\rho
    \end{equation*}
     for all  $i \in \{1,\dots,N\}$.  That is, the system is {\em permanent}.
     \label{cor:2}
\end{corollary}

\begin{proof}
	The lower bound follows by our assumption.  The upper bound follows from Lemmas \ref{lem:main2} (with $D$ equal to $P$ restricted to those $x$ a distance of at least $\delta$ away from the boundary), \ref{lem:no_union}, and similar arguments as in the proof of Theorem \ref{thm:main1}.  The only real difference in the proof is that the analog of equation \eqref{eq:good_bound} is 
\begin{equation*}
  \frac{\partial}{\partial t} V_1(\phi(t,x_0)) < -\epsilon,
  \end{equation*}
  whenever $|\phi(t,x_0)|>M$, giving us the needed force to guarantee $|\phi(t,x_0)|$ decreases below some $1/\rho$. 
\end{proof}

 Note that the $M=M_{\epsilon,\delta}>0$ of Lemma \ref{lem:main2}, and hence in the proof of Corollary \ref{cor:2}, explicitly depends upon $\delta$.  Therefore, it is not sufficient in the statement of Corollary \ref{cor:2} to assume the existence of a different $\delta = \delta_{x_0}>0$ for \textit{each} $x_0$.  Thus, the main results of \cite{Anderson2011} pertaining to weakly reversible networks (with arbitrary deficiency) are not strong enough to guarantee permanence using the above methods.

\bibliographystyle{amsplain} \bibliography{BoundedTraj}
  
\end{document}